\input louC.sty

\documentclass[a4paper,draft]{amsproc}
\usepackage{amsrefs}
\usepackage{amsmath}
\usepackage{mathrsfs}
\usepackage{amssymb}
\usepackage{amscd} 

\theoremstyle{plain}
 \newtheorem{thm}{\textbf{Theorem}}[section]
 
 \newtheorem{lem}{\textbf{Lemma}}[section]
 
\theoremstyle{definition}

\theoremstyle{remark}
 
 \numberwithin{equation}{section}

\renewcommand{\leq}{\leqslant}
\renewcommand{\geq}{\geqslant}
\renewcommand{\setminus}{\smallsetminus}
\title{Density of Polynomials in Sub-Bergman Hilbert Spaces}


\author[Chu]{\bfseries Cheng Chu}

\address{
Department of Mathematics \\ 
Vanderbilt University  \\ 
Nashville, Tennessee \\
USA}
\email{cheng.chu@vanderbilt.edu}

\begin{document}

\vspace{18mm}
\setcounter{page}{1}
\thispagestyle{empty}

\begin{abstract}
The sub-Bergman Hilbert spaces are analogues of de Branges-Rovnyak spaces in the Bergman space setting. We prove that the polynomials are dense in sub-Bergman Hilbert spaces. This answers the question posted by Zhu in the affirmative.
\end{abstract}

\maketitle

\section{Introduction}  

Let $A$ be a bounded operator on a Hilbert space $H$. We define the range space $\cM(A)=AH$, and endow it with the inner product
$$\langle Af, Ag \rangle_{\mathcal{M}(A)}=\langle f, g \rangle_{H},\qq f,g\in H \ominus \m{Ker}A.$$

Let $\DD$ denote the unit disk.
Let $L^2$ denote the Lebesgue space of square integrable functions on the unit circle $\partial\DD$. The Hardy space $H^2$ is the subspace of analytic functions on $\DD$ whose Taylor coefficients are square summable. Then it can also be identified with the subspace of $L^2$ of functions whose negative Fourier coefficients vanish.
The Toeplitz operator on the Hardy space $H^2$ with symbol $f$ in $L^\infty(\DD)$ is defined by
$$T_f (h) = P(fh),$$ for $h\in H^2(\DD)$. Here $P$ be the orthogonal projections from $L^2$ to $H^2$.

Let $b$ be a function in the closed unit ball of $H^\infty(\DD)$, the space of bounded analytic functions on the unit disk. The de Branges-Rovnyak space $\cH(b)$ is defined to be the space $(I-T_b T_{\bar b})^{1/2} H^2$.
We also define the space $\cH(\bar b)$ in the same way as $\cH(b)$, but with the roles of $b$ and $\bar{b}$ interchanged, i.e. $$\cH(\bar{b})=(I-T_{\bar b} T_b)^{1/2} H^2.$$ The spaces $\cH(b)$ and $\cH(\bar b)$ are also called sub-Hardy Hilbert spaces (the terminology comes from the title of Sarason's book \cite{sar94}).

The Bergman space $A^2$ is the space of analytic functions on $\DD$ that are square-integrable with respect to the normalized Lebesgue area measure $dA$. For $u\in L^\infty(\DD)$, the Bergman Toeplitz
operator $\tiT_u$ with symbol $u$ is the operator on $L^2_a$ defined by $$\tiT_u h = \tiP(uh).$$ Here $\tiP$ is the orthogonal projection from $L^2(\DD, dA)$ onto $A^2$.

In \cite{zhu96}, Zhu introduced the sub-Bergman Hilbert spaces. They are defined by
$$\cA(b)=(I-\tiT_b \tiT_{\bar b})^{1/2}A^2$$ and $$\cA(\bar{b})=(I-\tiT_{\bar b} \tiT_b)^{1/2}A^2.$$ Here $b$ is a function in the closed unit ball of $H^\infty(\DD)$. It is easy to see from the definition that the spaces $\cA(b)$ and $\cA(\bar{b})$ are contractively contained in $A^2$. But in most cases they are not closed subspaces of $A^2$ (see \cite{zhu96}*{Corollary 3.13}). Some examples of $\cA(b)$ are known: if $||b||_\infty<1$, then $\cA(b)$ is just a renormed version $A^2$; if $b$ is a finite Blaschke product, then $\cA(b)=H^2$ (\cite{zhu03}, \cite{sul06}).

Sub-Bergman Hilbert spaces share some common properties with sub-Hardy Hilbert spaces as the way those spaces are defined follows from a general theory on Hilbert space contractions \cite{sar94}. For instance, both $\cH(b)$ and $\cA(b)$ are invariant under the corresponding Toeplitz operators with a co-analytic symbol (\cite{sar94}*{II-7}).
One significant difference between the spaces $\cA(b)$ and $\cH(b)$ is the multipliers. The theory of $\cH(b)$ spaces is pervaded by a fundamental dichotomy, whether $b$ is an extreme point of the unit ball of $H^\infty(\DD)$. The multiplier structure of de Branges-Rovnyak spaces has been studied extensively by Lotto and Sarason in both the extreme and the nonextreme cases (\cite{lot90}, \cite{lotsar91}, \cite{lotsar93}). However, Zhu \cite{zhu96} showed that every function in $H^\infty$ is a multiplier of $\cA(b)$ and $\cA(\bar b)$. As a consequence, the two sub-Bergman Hilbert spaces $\cA(b)$ and $\cA(\bar{b})$ are norm equivalent.

\begin{thm}\label{=}
For $b\in H^\infty$ with $||b||_\infty\leq 1$,
$$\cA(b)=\cA(\bar{b}).$$
\end{thm}
\begin{proof}
The relation $$||f||_{\cA(b)}\leq||f||_{\cA(\bbb)}$$ follows from Douglas's criterion \cite{sar94}*{I-5} and the operator inequality (see e.g. \cite{zhu}*{p. 106})
$$\tiT_b\tiT_{\bbb}\leq \tiT_{\bbb}\tiT_b.$$
On the other hand, let $f\in \cA(\bbb)$. By \cite{zhu96}*{Theorem 3.12}, $b$ is a multiplier of $\cA(b)$ with multiplier norm less than or equal to $||b||_\infty$.
Then
$$
||f||^2_{\cA(\bbb)}=||f||^2_{A_2}+||bf||^2_{\cA(b)}\leq ||f||^2_{\cA(b)}+||b||^2_\infty||f||^2_{\cA(b)}\leq 2||f||^2_{\cA(b)}.
$$
Here we used the identity \cite{sar94}*{I-8}.
\end{proof}

In the Hardy space setting, the polynomials belong to $\cH(b)$ if and only if $b$ is non-extreme, and in this case polynomials are dense in $\cH(\bar b)$ and $\cH(b)$ (see \cite{sar94}*{Chapter IV, V}).  For the sub-Bergman Hilbert spaces,  Zhu in \cite{zhu96} showed that $\cA(b)$ always contain $H^\infty$, which includes the polynomials. In the same paper, Zhu asked whether the polynomials are dense in $\cA(b)$ or $\cA(\bbb)$. In this note, we answer this question in the affirmative.

\section{Proof the Main Result}

In the theory of $\cH(b)$ spaces, $\cH(\bar b)$ is often more amenable than $\cH(b)$ because of a representation theorem for $\cH(\bar b)$ \cite[III-2]{sar94}. We shall prove a similar version for $\cA(\bar b)$
\begin{thm}\label{rep}
Let $b\in H^\infty $ with $||b||_\infty\leq 1$. Let $L^2_b$ denote the space $L^2(\DD, (1-|b|^2)dA)$. Let $A^2_b$ be the closure of polynomials in $L^2_b$. Define the operator $S_b$ by
\beq\label{s}S_bg=\tiP((1-|b|^2) g).\eeq
 Then $S_b$ is an isometry from $A^2_b$ onto $\cA(\bar b)$.
\end{thm}
\begin{proof}
For any $q\in L_b^2$, $g\in A^2$, we have
$$
\la S_bq, g \ra_{A^2}=\la \tiP((1-|b|^2) q, g  \ra_{A^2}=\int_{\DD} q\bg (1-|b|^2)dA=\la q, g\ra_{L_b^2},
$$
which implies $S_b^*$ is the natural inclusion from $A^2$ into $L_b^2$. Then for any $f, g\in A^2$,
$$
\la S_bS_b^*f, g\ra_{A^2}=\la S_b^*f, S_b^* g\ra_{L_b^2}=\la f, g\ra_{L_b^2}=\int_{\DD} f\bg (1-|b|^2)dA=\la \tiT_{1-|b|^2}f, g\ra_{A^2}.
$$
Thus $$S_bS_b^*=I-\tiT_{\bar b} \tiT_b.$$ By Douglas's Criterion, $$\cA(\bar{b})=S_b L^2_b.$$
Notice that every $q\in L_b^2$, $$(S_b q)(w)=\la q, k_w\ra_{L_b^2},$$ where $k_w(z)=\frac{1}{(1-\bw z)^2}$ is the reproducing kernel of $A^2$.
We see that $$Ker (S_b)=(Span_{w\in \DD}\{k_w\})^{\perp}=(A_b^2)^{\perp}.$$
Hence $S_b$ is an isometry from $A^2_b$ onto $\cA(\bar b)$.
\end{proof}

We need the following lemma to prove the main result.

\begin{lem}\label{lem}
Let $b\in H^\infty $ with $||b||_\infty\leq 1$.
Let $M_n$ denote the closure of the linear span of $\{z^k\}_{k=n}^{\infty}$ in $A_b^2$. Let $\cP=\cup_{n\geq 0}(M_n)^{\perp}.$ Then $\cP$ is dense in $A_b^2$.
\end{lem}
\begin{proof}
If $b$ is a constant, the conclusion is obvious. Next, we assume $b$ is not a constant.
Let $f\in A_b^2$ such that $f\perp \cP$. Then $f\in M_n$ for every $n\geq 0$. It is sufficient to show $f=0$.

Suppose $f \in A_{b}^{2} \setminus \{0\}$. Then for some $k \geq 1$, $f$ has the power series expansion
$$f(z) = \sum_{j = k}^{\infty} a_j z^j$$
with $a_k \not = 0$. Since $f \in M_k$, there is a sequence of polynomials $\{p_s\}_{s=1}^{\infty} \subset M_{k}$ for which $$p_s \to f \q\m{in}\q L_{b}^{2}.$$ By the definition of the integral norm,
$$p_s - a_k z^k \to f - a_k z^k \q\m{in}\q L_{b}^{2},$$ and thus $f - a_{k} z^k \in M_{k + 1}$. But since $f \in M_{k + 1}$ and $a_k \not = 0$, it must be the case that $z^k \in M_{k + 1}$.

Fix $r\in (0, 1)$. There exists $\Gd>0$ such that $$1-|b(z)|^2\geq \Gd,$$ for every $z$ in the disk $r\DD$.

Let $$g(z)=\sum_{j=k+1}^{N} b_jz^j, N>k$$ be a polynomial with minimal degree at least $k+1$.
Using the easily established identity
$$
\int_{r\DD} z^m\bar{z^n}dA(z)=
\begin{cases}\frac{r^{2m+2}}{m+1}, & m=n, \cr
0, &m\neq n,
\end{cases}
$$
we have
\begin{align*}
||z^k-g(z)||^2_{A_b^2}&=\int_{\DD} |z^k-g(z)|^2(1-|b(z)|^2)dA(z)\\
&\geq \int_{r\DD} |z^k-g(z)|^2(1-|b(z)|^2)dA(z)\\
&\geq \Gd\int_{r\DD} |z^k-g(z)|^2dA(z)\\
&=\Gd\large(\int_{r\DD} |z^k|^2dA(z)+\int_{r\DD} |g(z)|^2dA(z)\large)\\
&\geq\frac{\Gd }{k+1}r^{2k+2}>0.
\end{align*}
By the definition of $M_{k+1}$, we see that $z^k\notin  M_{k+1}$, which is a contradiction . Thus $f=0$.
\end{proof}

\begin{thm}
Polynomials are dense in $\cA(b)$.
\end{thm}
\begin{proof}
By Theorem \ref{=}, it suffices to prove polynomials are dense in $\cA(\bar{b})$. Let $f\in \cA(\bbb)$ and let $\Gve>0$. By Theorem \ref{rep}, there exists $g\in A_b^2$ such that $f=S_b g$.
Using Lemma \ref{lem}, we can find $h\in \cP$ with $$||g-h||_{A_b^2}<\Gve.$$
By Theorem \ref{rep}, $S_b: A_b^2\to \cA(\bbb)$ is an isometry. Thus $$||f-S_bh||_{\cA(\bbb)}=||S_b(g-h)||_{\cA(\bbb)}=||g-h||_{A_b^2}<\Gve.$$

It remains to show that $S_bh$ is a polynomial. From the definition of $\cP$, we see that $h\in (M_n)^\perp$, for some $n\geq 0$. Then $$\la h, z^k\ra_{A_b^2}=0,$$ for every $k\geq n$.
If $n=0$, then $h=0$. If $n\geq 1$, using \eqref{s} we have
\begin{align*}
\la S_bh, z^k\ra_{A^2}&=\la\tiP((1-|b|^2)h), z^k\ra_{A_b^2}=\int_{\DD} h(z)\bz^k(1-|b(z)|^2)dA(z)\\
&=\la h, z^k\ra_{A_b^2}=0,
\end{align*}
for every $k\geq n.$ Therefore $S_bh$ is a polynomial of degree at most $n-1$.
\end{proof}

\bibliography{references}
\end{document}